\documentclass[12pt,reqno]{amsproc}
\usepackage{amsfonts}
\usepackage{amsmath}
\usepackage{amssymb}
\usepackage{graphicx}

 \theoremstyle{plain}
\newtheorem{theorem}{Theorem}[section]

\newtheorem{corollary}[theorem]{Corollary}

\newtheorem{example}[theorem]{Example}

\newtheorem{lemma}[theorem]{Lemma}

\numberwithin{equation}  {section}
\begin{document}
	\title[New Chen-Beurling Theorem]{An Extension of the
		Chen-Beurling-Helson-Lowdenslager Theorem}
	\author{Haihui Fan}
	\curraddr{University of New Hampshire}
	\email{hun4@wildcats.unh.edu}
	\author{Don Hadwin}
	\curraddr{University of New Hampshire}
	\email{don@unh.edu}
	\author{Wenjing Liu}
	\address{University of New Hampshire}
	\email{wbs4@wildcats.unh.edu}
	\thanks{}
	\dedicatory{Dedicated to Peter Rosenthal, a great mathematician, a great
		lawyer, and a great human being\\
	}

\begin{abstract}
Yanni Chen \cite{chen2} extended the classical Beurling-Helson-Lowdenslager
Theorem for Hardy spaces on the unit circle $\mathbb{T}$ defined in terms of
continuous gauge norms on $L^{\infty}$ that dominate $\Vert\cdot\Vert_{1}$. We
extend Chen's result to a much larger class of continuous gauge norms. A key
ingredient is our result that if $\alpha$ is a continuous normalized gauge
norm on $L^{\infty}$, then there is a probability measure $\lambda$, mutually
absolutely continuous with respect to Lebesgue measure on $\mathbb{T}$, such
that $\alpha\geq c\Vert\cdot\Vert_{1,\lambda}$ for some $0<c\leq1.$

\end{abstract}

\maketitle
\section{Introduction}

Let $\mathbb{T}$ be the unit circle, i.e., $\mathbb{T}=\{ \lambda\in
\mathbb{C}:\left\vert \lambda\right\vert =1\}$, and let $\mu$ be Haar measure
(i.e., normalized arc length) on $\mathbb{T}$. The classical and influential
Beurling-Helson-Lowdenslager theorem (see \cite{B},\cite{H.L}) states that if $W$ is a closed
$H^{\infty}(\mathbb{T},\mu)$-invariant subspace (or, equivalently,
$zW\subseteq W)$ of $L^{2}\left(  \mathbb{T},\mu\right)  \text{,}$ then
$W=\varphi H^{2}$ for some $\varphi\in L^{\infty}(\mathbb{T},\mu)\text{,}$
with $\left\vert \varphi\right\vert =1$ a.e.$(\mu)$ or $W=\chi_{E}%
L^{2}(\mathbb{T},\mu)$ for some Borel set $E\subset\mathbb{T}$. If $0\neq
W\subset H^{2}(\mathbb{T},\mu)$, then $W=\varphi H^{2}(\mathbb{T},\mu)\;$for
some $\varphi\in H^{\infty}(\mathbb{T},\mu)$ with $\left\vert \varphi
\right\vert =1$ a.e. $(\mu)$. Later, the Beurling's theorem was extended to
$L^{p}(\mathbb{T},\mu)$ and $H^{p}(\mathbb{T},\mu)$ with $1\leq p\leq\infty$,
with the assumption that $W$ is weak*-closed when $p=\infty$ (see
\cite{halmos},\cite{helson},\cite{H.L},\cite{hoffman}). In \cite{chen2}, Yanni
Chen extended the Helson-Lowdenslager-Beurling theorem for all continuous
$\Vert\cdot\Vert_{1,\mu}$-dominating normalized gauge norms on $\mathbb{T}$.

In this paper we extend the Helson-Lowdenslager-Beurling theorem for a much
larger class of norms. We first extend Chen's results to the case of
$c\Vert\cdot\Vert_{1,\mu}$-dominating continuous gauge norms. We then prove
that for any continuous gauge norm $\alpha$, there is a probability measure
$\lambda$ that is mutually absolutely continuous with respect to $\mu$ such
that $\alpha$ is $c\Vert\cdot\Vert_{1,\lambda}$-dominating. We use this result
to extend Chen's theorem. Our extension depends on Radon-Nikodym derivative
$d\lambda/d\mu$. In particular, Chen's theorem extends exactly whenever
$\log\left(  d\lambda/d\mu\right)  \in L^{1}(\mathbb{T},\mu)$.

\section{Continuous Gauge Norms on $\Omega$}

Suppose $\left(  \Omega,\Sigma,\nu\right)  $ is a probability space. A norm
$\alpha$ on $L^{\infty} (\Omega,\nu)$ is a \emph{normalized gauge norm} if

\begin{enumerate}
\item $\alpha(1) =1\text{,}$

\item $\alpha(\vert f\vert) =\alpha(f)$ for every $f \in L^{\infty}
(\Omega,\nu)$.
\end{enumerate}

In addition we say $\alpha$ is \emph{continuous} ($\nu$\emph{-continuous}) if%

\[
\lim_{\nu(E) \rightarrow0}\alpha(\chi_{E}) =0 ,\; \text{}\;
\]
that is, whenever $\{E_{n}\}$ is a sequence in $\Sigma$ and $\nu\left(
E_{n}\right)  \rightarrow0$, we have $\alpha\left(  \chi_{E_{n}}\right)
\rightarrow0$.

We say that a\textbf{\ }\emph{normalized gauge norm} $\alpha$ is $c \Vert
\cdot\Vert_{1 ,\nu}$\emph{-dominating} for some $c >0$ if%

\[
\alpha(f) \geq c \Vert f\Vert_{1 ,\nu} ,\; \text{for every}\;f \in L^{\infty}
(\Omega,\nu)\text{.}%
\]

It is easily to see the following fact that\newline(1) The common norm
$\Vert\cdot\Vert_{p,\nu}$ is a $\alpha$ norm for $1\leq p\leq\infty$.
\newline(2) If $\nu$ and $\lambda$ are mutually absolutely continuous
probability measures, then $L^{\infty}(\Omega,\nu)=L^{\infty}(\Omega,\lambda)$
and a normalized gauge norm is $\nu$-continuous if and only if it is $\lambda$-continuous.

We can extend the normalized gauge norm $\alpha$ from $L^{\infty}(\Omega,\nu)$
to the set of all measurable functions, and define $\alpha$ for all measurable
functions $f$ on $\Omega$ by%
\[
\alpha(f)=\sup\{ \alpha(s):s\; \text{is a simple function}\;,0\leq s\leq|f|\}
\text{.}%
\]
It is clear that $\alpha(f)=\alpha(|f|)$ still holds. \newline Define
\[
\mathcal{L}^{\alpha}(\Omega,\nu)=\{f:f\; \text{is a measurable function on}\;
\Omega\; \; \text{with}\; \alpha(f)<\infty\} \text{,}%
\]%
\[
L^{\alpha}(\Omega,\nu)=\overline{L^{\infty}(\nu)}^{\alpha},\; \text{i.e.,
the}\; \alpha\; \text{-closure of}\;L^{\infty}(\nu)\; \text{in}\;
\mathcal{L}^{\alpha}\text{.}%
\]

Since $L^{\infty} \left(  \Omega,\nu\right)  $ with the norm $\alpha$ is dense
in $L^{\alpha} (\Omega,\nu)$, they have the same dual spaces. We prove in the
next lemma that the normed dual $\left(  L^{\alpha} (\Omega,\nu)
,\alpha\right)  ^{\#} =\left(  L^{\infty} \left(  \Omega,\nu\right)
,\alpha\right)  ^{\#}$ can be viewed as a vector subspace of $L^{1}
(\Omega,\nu)$. Suppose $w \in L^{1} (\Omega,\nu)$, we define the functional
$\varphi_{w} :L^{\infty} (\Omega,\nu) \rightarrow\mathbb{C}$ by%
\[
\varphi_{w} \left(  f\right)  =\int_{\Omega}f w d\nu\text{.}%
\]

\begin{lemma}
\label{dual}Suppose $\left(  \Omega,\Sigma,\nu\right)  $ is a probability
space and $\alpha$ is a continuous normalized gauge norm on $L^{\infty}%
(\Omega,\nu)$. Then \newline(1) if $\varphi:L^{\infty}(\Omega,\nu
)\rightarrow\mathbb{C}$ is an $\alpha$-continuous linear functional, then
there is a $w\in L^{1}(\Omega,\nu)$ such that $\varphi=\varphi_{w}$,
\newline(2) if $\varphi_{w}$ is $\alpha$-continuous on $L^{\infty}(\Omega
,\nu)$, then \newline(a) $\left\Vert w\right\Vert _{1,\mu}\leq\left\Vert
\varphi_{w}\right\Vert =\left\Vert \varphi_{\left\vert w\right\vert
}\right\Vert \text{,}$ \newline(b) given $\varphi\;$in the dual of $L^{\alpha
}(\Omega,\lambda)$, i.e., $\varphi\in\left(  L^{\alpha}(\Omega,\lambda
)\right)  ^{\#}\text{,}$ there exists a $w\in L^{1}(\Omega,\lambda)$, such
that \newline%
\[
\forall f\in L^{\infty}(\Omega,\lambda),\varphi(f)=\int_{\Omega}%
fwd\lambda\; \text{and}\;wL^{\alpha}(\Omega,\lambda)\subseteq L^{1}%
(\Omega,\lambda)\text{.}%
\]

\end{lemma}

\begin{proof}
(1) If $\alpha$ is continuous, it follows that, whenever $\left\{
E_{n}\right\}  $ is a disjoint sequence of measurable sets,%
\[
\lim_{N\rightarrow\infty}\alpha\left(  \chi_{\cup_{n=1}^{\infty}E_{n}}%
-\sum_{k=1}^{N}\chi_{E_{k}}\right)  =\lim_{N\rightarrow\infty}\alpha\left(
\chi_{\cup_{k=N+1}^{\infty}E_{k}}\right)  =0\text{,}%
\]
since $\lim_{N\rightarrow\infty}\nu\left(  \cup_{k=N+1}^{\infty}E_{k}\right)
=0.$ It follows that%
\[
\rho\left(  E\right)  =\varphi\left(  \chi_{E}\right)
\]
defines a measure $\rho$ and $\rho<<\nu$. It follows that if $w=d\rho
/d\nu\text{,}$ then%
\[
\left\Vert w\right\Vert _{1,\nu}=\sup\left\{  \left\vert \int_{\Omega}%
wsd\nu\right\vert :s\; \text{is simple,}\; \left\Vert s\right\Vert _{\infty
}\leq1\right\}
\]%
\[
=\sup\left\{  \left\vert \varphi\left(  s\right)  \right\vert :s\;
\text{simple,}\; \left\Vert s\right\Vert _{\infty}\leq1\right\}
\leq\left\Vert \varphi\right\Vert \text{.}%
\]
Hence $w\in L^{1}(\Omega,\nu)$. Also, since, for every $f\in L^{\infty}%
(\Omega,\nu)$%
\[
\left\vert \varphi\left(  f\right)  \right\vert \leq\left\Vert \varphi
\right\Vert \alpha\left(  f\right)  \leq\left\Vert \varphi\right\Vert
\left\Vert f\right\Vert _{\infty}\text{,}%
\]
we see that $\varphi$ is $\Vert\cdot\Vert_{\infty}$-continuous on $L^{\infty
}(\Omega,\nu)\text{,}$ so it follows that $\varphi=\varphi_{w}$.

(2a) From (1) we will see $\left\Vert w\right\Vert _{1,\nu}\leq\left\Vert
\varphi\right\Vert .$

(2b) For any measurable set $E\subseteq\Omega$, and for all $\varphi\in\left(
L^{\alpha}(\lambda)\right)  ^{\#}$, define $\rho(E)=\varphi(\chi_{E})\text{.}$
we can prove $\rho$ is a measure as in theorem \ref{thm(1)}, and $\rho
\ll\lambda$. By Radon-Nikodym theorem, there exists a function $w\in
L^{1}(\lambda)$ such that, for every measurable set $E\subseteq\Omega$,
$\varphi(\chi_{E})=\rho(E)=\int_{\Omega}\chi_{E}wd\lambda$. Thus $\forall f\in
L^{\infty}(\Omega,\lambda)$, $\varphi(f)=\int_{\Omega}fwd\lambda=\int_{\Omega
}fwgd\mu=\int_{\Omega}fw|h|d\mu=\int_{\Omega}fwuhd\mu=\int_{\Omega
}f\widetilde{w}hd\mu$, where $\widetilde{w}=wu,|\widetilde{w}|=|w|\text{,}$
here $\widetilde{w}\in L^{1}(\Omega,\lambda)$ and $g,h$ as in theorem
\ref{thm(1)}, so $\widetilde{w}h\in L^{1}(\mu)$. Therefore, $\varphi
(f)=\int_{\Omega}f\widetilde{w}hd\mu$ for all $f\in L^{\alpha}(\Omega
,\lambda)$.

Suppose $f\in L^{\alpha}(\Omega,\lambda)$, $f=u|f|,|u|=1$. $|f|\in L^{\alpha
}(\Omega,\lambda)$. There exists an increasing positive sequence $s_{n}$ such
that $s_{n}\rightarrow|f|$ a.e.$(\mu)$, thus $us_{n}\rightarrow u|f|$
a.e.$(\mu)$. $\forall w\in L^{1}(\Omega,\lambda),w=v|w|$, where $|v|=1$, so we
have $\overline{v}s_{n}\rightarrow\overline{v}|f|$ a.e.$(\mu)$, where
$\overline{v}$ is the conjugate of $v$ and $\alpha(\overline{v}s_{n}%
-\overline{v}|f|)\rightarrow0$. Thus $\varphi(\overline{v}s_{n})\rightarrow
\varphi(\overline{v}|f|)$. On the other hand, we also have $\varphi
(\overline{v}s_{n})=\int_{\Omega}\overline{v}s_{n}wd\lambda\rightarrow
\int_{\Omega}\overline{v}|f|wd\lambda=\int_{\Omega}|f||w|d\lambda$ by monotone
convergence theorem. Thus $\int_{\Omega}|f||w|d\lambda=\int_{\Omega
}|f|\overline{v}wd\lambda=\varphi(\overline{v}|f|)<\infty$, therefore $fw\in
L^{1}(\Omega,\lambda)$, i.e., $wL^{\alpha}(\Omega,\lambda)\subseteq
L^{1}(\Omega,\lambda)$, where $w\in L^{1}(\Omega,\lambda)$.
\end{proof}

\begin{theorem}
\label{thm(1)}Suppose $\left(  \Omega,\Sigma,\nu\right)  $ is a probability
space, $\alpha$ is a continuous normalized gauge norm on $L^{\infty}%
(\Omega,\nu)$ and $\varepsilon>0$. Then there exists a constant $c$ with
$1-\varepsilon<c\leq1$ and a probability measure $\lambda$ on $\Sigma$ that is
mutually absolutely continuous with respect to $\nu$ such that $\alpha$ is
$c\Vert\cdot\Vert_{1,\lambda}$-dominating.
\end{theorem}

\begin{proof}
Let $M=\{ \nu\left(  h^{-1}\left(  \left(  0,\infty\right)  \right)  \right)
:h\in L^{1}(\Omega,\nu),h\geq0,\varphi_{h}$ is $\alpha$-continuous$\}
\text{.}$ It follows from lemma \ref{dual} that $M\neq\varnothing$. Choose
$\left\{  h_{n}\right\}  $ in $L^{1}(\Omega,\nu)$ such that $h_{n}%
\geq0\text{,}$ $\varphi_{h_{n}}$ is $\alpha$-continuous, and such that%
\[
\nu\left(  h_{n}^{-1}\left(  \left(  0,\infty\right)  \right)  \right)
\rightarrow\sup M\text{.}%
\]
Let
\[
h_{0}=\sum_{n=1}^{\infty}\frac{1}{2^{n}}\frac{1}{\left\Vert \varphi_{h_{n}%
}\right\Vert }h_{n}\text{.}%
\]
Since $\left\Vert h_{n}\right\Vert _{1,\nu}\leq\left\Vert \varphi_{h_{n}%
}\right\Vert \text{,}$ we see that $\left\Vert h_{0}\right\Vert _{1,\nu}%
\leq1.$ Also
\[
\varphi_{h_{0}}=\sum_{n=1}^{\infty}\frac{1}{2^{n}}\frac{1}{\left\Vert
\varphi_{h_{n}}\right\Vert }\varphi_{h_{n}}\text{,}%
\]
so $\varphi_{h_{0}}$ is $\alpha$-bounded and $\left\Vert \varphi_{h_{0}%
}\right\Vert \leq1$. On the other hand $h_{n}^{-1}\left(  \left(
0,\infty\right)  \right)  \subset h_{0}^{-1}\left(  \left(  0,\infty\right)
\right)  $ for $n\geq1\text{,}$ so we have%
\[
\nu\left(  h_{0}^{-1}\left(  \left(  0,\infty\right)  \right)  \right)  =\sup
M\text{.}%
\]
Let $E=\Omega\backslash h_{0}^{-1}\left(  \left(  0,\infty\right)  \right)  $
and assume, via contradiction, that $\nu\left(  E\right)  >0$. Then
$\alpha\left(  \chi_{E}\right)  >0$. Hence, by the Hahn-Banach theorem, there
is a $g\in L^{1}(\Omega,\nu)$ such that $\left\Vert \varphi_{g}\right\Vert =1$
and
\[
\alpha\left(  \chi_{E}\right)  =\varphi_{g}\left(  \chi_{E}\right)
=\int_{\Omega}g\chi_{E}d\nu=\varphi_{g\chi_{E}}\left(  \chi_{E}\right)
\leq\varphi_{\left\vert g\right\vert \chi_{E}}\left(  \chi_{E}\right)
\text{.}%
\]
It follows that $\nu\left(  \left(  \left\vert g\right\vert \chi_{E}\right)
^{-1}\left(  0,\infty\right)  \right)  =\eta>0\text{,}$ and that if
$h_{1}=h_{0}+\left\vert g\right\vert \chi_{E}$, then
\[
\sup M\geq\nu\left(  h_{1}^{-1}\left(  \left(  0,\infty\right)  \right)
\right)  =\nu\left(  h^{-1}\left(  \left(  0,\infty\right)  \right)  \right)
+\eta=\sup M+\eta\text{.}%
\]
This contradiction shows that $\nu\left(  E\right)  =0$, so we can assume that
$h_{0}\left(  \omega\right)  >0$ a.e. $\left(  \nu\right)  $. By replacing
$h_{0}$ with $h_{0}/\int_{\Omega}h_{0}d\nu$, we can assume that $\int_{\Omega
}h_{0}d\nu=1$.

If we define a probability measure $\lambda:\Sigma\rightarrow\left[
0,1\right]  $ by%
\[
\lambda\left(  E\right)  =\int_{E}h_{0}d\nu\text{,}%
\]
then $\lambda$ is a measure, $\lambda<<\nu$ and $\nu<<\lambda$ since $0<h_{0}$
a.e. $\left(  \nu\right)  $. Also, we have for every $f\in L^{\infty}\left(
\Omega,\nu\right)  \text{,}$%
\[
\left\Vert f\right\Vert _{1,\lambda}=\int_{\Omega}\left\vert f\right\vert
d\lambda=\int_{\Omega}\left\vert f\right\vert h_{0}d\nu=\varphi_{h_{0}}\left(
\left\vert f\right\vert \right)  \leq\left\Vert \varphi_{h_{0}}\right\Vert
\alpha\left(  f\right)  \text{.}%
\]
Since $\varphi_{h_{0}}\left(  1\right)  =1$, we know $\left\Vert
\varphi_{h_{0}}\right\Vert \geq1.$ Hence, $0<c_{0}=1/\left\Vert \varphi
_{h_{0}}\right\Vert \leq1\text{,}$ and we see that $\alpha$ is $c_{0}%
\Vert\cdot\Vert_{1,\lambda}$-dominating on $E$. If we apply the Hahn-Banach
theorem as above with $E=\Omega$, we can find a nonnegative function $k\in
L^{1}(\Omega,\nu)$ such that
\[
\left\Vert \varphi_{k}\right\Vert =1=\alpha\left(  1\right)  =\varphi
_{k}\left(  1\right)  =\int_{\Omega}k1d\nu\text{.}%
\]
For $0<t<1$ let $h_{t}=\left(  1-t\right)  k+th_{0}$. Then $\varphi_{h_{t}%
}=\left(  1-t\right)  \varphi_{k}+t\varphi_{h_{0}}$. Thus
\[
\lim_{t\rightarrow0^{+}}\left\Vert \varphi_{h_{t}}\right\Vert =\left\Vert
\varphi_{k}\right\Vert =1.
\]
Choose $t$ so that $\left\Vert \varphi_{h_{t}}\right\Vert <1/\left(
1-\varepsilon\right)  $, so $1-\varepsilon<c=1/\left\Vert \varphi_{h_{t}%
}\right\Vert \leq1$. If we define a probability measure $\lambda_{t}%
:\Sigma\rightarrow\left[  0,1\right]  $ by%
\[
\lambda_{t}\left(  E\right)  =\int_{E}h_{t}d\nu\text{,}%
\]
we see that $\lambda_{t}\;<<\mu\nu$ and since $h_{t}\geq th_{0}>0$, we see
$\nu<<\lambda_{t}$. As above we see, for every $f\in L^{\infty}(\Omega,\mu)$
we have%
\[
c\left\Vert f\right\Vert _{1,\lambda_{t}}\leq\frac{1}{\left\Vert
\varphi_{h_{t}}\right\Vert }\int_{\Omega}\left\vert f\right\vert h_{t}%
d\nu=\frac{1}{\left\Vert \varphi_{h_{t}}\right\Vert }\varphi_{h_{t}}\left(
\left\vert f\right\vert \right)  \leq\alpha\left(  f\right)  \text{.}%
\]
Therefore, $\alpha$ is $c\Vert\cdot\Vert_{1,\lambda_{t}}$-dominating on
$\Omega\text{.}$
\end{proof}

If we take $\Omega=\mathbb{T}\text{,}$ theorem \ref{thm(1)} holds for the
probability space $\left(  \Omega,\nu\right)  =$ $(\mathbb{T},\mu)\text{.}$
The $L^{p}$-version of the Helson-Lowdenslager theorem also holds, in a sense,
on the circle $\mathbb{T}$ when $\mu$ is replaced with a mutually absolutely
continuous probability measure $\lambda\text{.}$ Here the role of
$H^{p}\left(  \mathbb{T},\lambda\right)  $ is replaced with $\left(
1/g^{\frac{1}{p}}\right)  H^{p}\left(  \mathbb{T},\mu\right)  $. This result
is well-known, we include a proof for completeness as the following corollary.

\begin{corollary}
\label{unimodu}Suppose $\lambda$ is a probability measure on $\mathbb{T}$ and
$\mu< <\lambda$ and $\lambda< <\mu$. Let $g =d \lambda/d \mu$ and suppose $1
\leq p <\infty$. Suppose $W$ is a closed subspace of $L^{p} (\mathbb{T}
,\lambda)$, and $z W \subset W$. Then $g^{\frac{1}{p}} W =\chi_{E} L^{1}
(\mathbb{T} ,\mu)$ for some Borel subset $E$ of $\mathbb{T}$ or $g^{\frac
{1}{p}} W =\varphi H^{p} (\mathbb{T} ,\mu)$ for some unimodular function
$\varphi$.
\end{corollary}

\begin{proof}
Define $U :L^{p} (\mathbb{T} ,\lambda) \longrightarrow L^{p} (\mathbb{T}
,\mu)$ by $U f =f g^{\frac{1}{p}}$, for $f \in L^{p} (\mathbb{T} ,\lambda)$.
Clearly $U$ is a surjective isometry, since%
\[
\Vert U f\Vert_{p ,\mu}^{p} =\int_{\mathbb{T}}\left\vert f g^{\frac{1}{p}%
}\right\vert ^{p} d\mu=\int_{\mathbb{T}}\left\vert f\right\vert ^{p} g
d\mu=\int_{\mathbb{T}}\left\vert f\right\vert ^{p} d\lambda=\Vert f\Vert_{p
,\lambda}\text{.}%
\]
Define%
\[
M_{z ,\mu} :L^{p} (\mathbb{T} ,\mu) \longrightarrow L^{p} (\mathbb{T} ,\mu)\;
\text{by}\;M_{z ,\mu} f =z f\; \text{and}\;M_{z ,\lambda} :L^{p} (\mathbb{T}
,\lambda) \longrightarrow L^{p} (\mathbb{T} ,\lambda)\; \text{by}\;M_{z
,\lambda} f =z f\text{.}%
\]
Then
\[
U M_{z ,\lambda} f =U (z f) =g^{\frac{1}{p}} z f =z g^{\frac{1}{p}} f =M_{z
,\mu} g^{\frac{1}{p}} f =M_{z ,\mu} U f\text{,}%
\]
so $U M_{z ,\lambda} =M_{z ,\mu} U\text{.}$ It follows that $W$ is a closed
$z$-invariant subspace of $L^{p} (\mathbb{T} ,\lambda)$ if and only if
$g^{\frac{1}{p}} W =U (W)$ is a $z$-invariant closed linear subspace of $L^{p}
(\mathbb{T} ,\mu)$. The conclusion now follows from the classical Beurling
theorem for $L^{p} \left(  \mathbb{T} ,\mu\right)  $.
\end{proof}

\section{Continuous Gauge Norms on the Unit Circle}

Suppose $\alpha$ is a continuous normalized gauge norm on $L^{\infty}\left(
\mathbb{T},\mu\right)  $, suppose that $c>0$ and $\lambda$ is a probability
measure on $\mathbb{T}$ such that $\lambda<<\mu$ and $\mu<<\lambda$ and such
that $\alpha$ is $c\Vert\cdot\Vert_{1,\lambda}$-dominating. We let
$g=d\lambda/d\mu$ and $g>0$. We consider two cases \newline(1) $\int\left\vert
\log g\right\vert d\mu<\infty\text{,}$ \newline(2) $\int\left\vert \log
g\right\vert d\mu=\infty\text{.}$

We define $L^{p}\left(  \mathbb{T},\lambda\right)  $ to be the $\Vert
\cdot\Vert_{p,\lambda}$-closure of $L^{\infty}\left(  \mathbb{T}%
,\lambda\right)  $ and define $H^{p}(\mathbb{T},\lambda)$ to be $\Vert
\cdot\Vert_{p,\lambda}$-closure of the polynomials for $1\leq p<\infty$.
Denote $L^{\infty}(\mathbb{T},\mu)=L^{\infty}(\mu)\text{,}$ $L^{p}%
(\mathbb{T},\mu)=L^{p}(\mu)$ and $H^{p}(\mathbb{T},\mu)=H^{p}(\mu)\text{.}$

\begin{lemma}
\label{cases}The following are true:\newline(1) $\int\left\vert \log
g\right\vert d\mu<\infty$ $\Leftrightarrow$ there is an outer function $h\in
H^{1}\left(  \mu\right)  $ with $\left\vert h\right\vert =g$, \newline(2)
$\int\left\vert \log g\right\vert d\mu=\infty$ $\Leftrightarrow$ $H^{1}\left(
\lambda\right)  =L^{1}\left(  \lambda\right)  $.
\end{lemma}

\begin{proof}
Clearly $H^{1} \left(  \lambda\right)  $ is a closed $z$-invariant subspace of
$L^{1} \left(  \lambda\right)  $. Thus, by corollary \ref{unimodu}, either $g
H^{1} \left(  \lambda\right)  =\varphi H^{1} \left(  \mu\right)  $ for some
unimodular $\varphi$ or $g H^{1} \left(  \lambda\right)  =\chi_{E} L^{1}
\left(  \mu\right)  $ for some Borel set $E \subset\mathbb{T}$.

For (1), if $gH^{1}(\lambda)=\varphi H^{1}(\mu)$ for some unimodular
$\varphi\text{,}$ and $0<g\in gH^{1}(\lambda)\text{,}$ then $0\neq
\overline{\varphi}g\in H^{1}(\mu)$ which implies $\log g=\log\left\vert
\overline{\varphi}g\right\vert \in L^{1}(\mu).$ It is a standard fact that if
$g>0$ and $\log g$ are in $L^{1}(\mu),$ then there exists an outer function
$h\in H^{1}(\mu)$ with the same modulus as $g,$(i.e., $\left\vert h\right\vert
=g\text{).}$ Therefore, (1) is proved by lemma 3.2 in \cite{chen2}.

For (2), Since $gH^{1}(\lambda)=\varphi H^{1}(\mu)$ if and only if
$\int\left\vert \log g\right\vert d\mu<\infty.$Suppose $\int\left\vert \log
g\right\vert d\mu=\infty\text{,}$ then $gH^{1}\left(  \lambda\right)
=\chi_{E}L^{1}\left(  \mu\right)  \text{.}$ We have $g=\chi_{E}f$ for some
$f\in L^{1}(\mu)\text{,}$ which implies $\chi_{E}=1$ since $g>0.$ Thus
$gH^{1}\left(  \lambda\right)  =L^{1}\left(  \mu\right)  =gL^{1}(\mu)\text{,}$
which implies $H^{1}(\lambda)=L^{1}(\lambda)\text{.}$ Conversely, if
$H^{1}(\lambda)=L^{1}(\lambda)\text{,}$ then $gH^{1}\left(  \lambda\right)
=gL^{1}\left(  \lambda\right)  =L^{1}\left(  \mu\right)  =\chi_{\mathbb{T}%
}L^{1}(\mu)\text{,}$ which means $gH^{1}(\lambda)\neq\varphi H^{1}%
(\mu)\text{,}$ i.e., $\int\left\vert \log g\right\vert d\mu=\infty\text{.}$
\end{proof}

There is an important characterization of outer functions in $H^{1} \left(
\mu\right)  $.

\begin{lemma}
\label{outer}A function $f$ is an outer function in $H^{1}\left(  \mu\right)
$ if and only there is a real harmonic function $u$ with harmonic conjugate
$\overline{u}$ such that \newline(1) $u\in L^{1}\left(  \mu\right)  \text{,}$
\newline(2) $f=e^{u+i\overline{u}}\text{,}$ \newline(3) $f\in L^{1}\left(
\mu\right)  \text{.}$
\end{lemma}

\textbf{Through the remainder of following sections we assume }

\begin{enumerate}
\item $\alpha$ is a continuous normalized gauge norm on $L^{\infty} \left(
\mu\right)  \text{.}$

\item and that $c >0$ and $\lambda$ is a probability measure on $\mathbb{T}$
such that $\lambda< <\mu$ and $\mu< <\lambda$ and such that $\alpha$ is $c
\Vert\cdot\Vert_{1 ,\lambda}$ dominating.

\item $h \in H^{1} \left(  \mu\right)  $ is an outer function, $\eta$ is
unimodular and $\bar{\eta} h =g =d \lambda/d \mu$.
\end{enumerate}

Since $\lambda$ and $\mu$ are mutually absolutely continuous we have
$L^{\infty} (\mu) =L^{\infty} (\lambda) ,L^{\alpha} (\mu) =$ $L^{\alpha}
(\lambda)$ and $H^{\alpha} (\mu) =$ $H^{\alpha} (\lambda)\text{,}$ we will use
$L^{\infty}$ to denote $L^{\infty} (\mu)$ and $L^{\infty} (\lambda)$, use
$L^{\alpha}$ to denote $L^{\alpha} (\mu)$ and $L^{\alpha} (\lambda)$, use
$H^{\alpha}$ to denote $H^{\alpha} (\mu)$ and $H^{\alpha} (\lambda)\text{.}$
It follows that $L^{\alpha} ,L^{\infty} ,H^{\alpha}$ do not depend on
$\lambda$ or $\mu$. However, this notation slightly conflicts with the
classical notation for $L^{1} \left(  \mu\right)  =L^{\Vert\cdot\Vert_{1 ,\mu
}}$ or $H^{1} \left(  \mu\right)  =H^{\Vert\cdot\Vert_{1 ,\mu}}$, so we will
add the measure to the notation when we are talking about $L^{p}$ or $H^{p}$.

\begin{theorem}
\label{hlmd}We have $h L^{1} (\lambda) =L^{1} (\mu)$ and $h H^{1} (\lambda)
=H^{1} (\mu)\text{.}$
\end{theorem}

\begin{proof}
We know from our assumption (3) that $hL^{1}(\lambda)=g\eta L^{1}\left(
\lambda\right)  =gL^{1}\left(  \lambda\right)  =L^{1}\left(  \mu\right)
\text{. }$By lemma \ref{cases}(1), we have $gH^{1}\left(  \lambda\right)
=\eta H^{1}\left(  \mu\right)  \text{,}$ so
\[
hH^{1}\left(  \lambda\right)  =\eta gH^{1}\left(  \lambda\right)  =\eta\eta
H^{1}\left(  \mu\right)  =H^{1}\left(  \mu\right)  \text{.}%
\]

\end{proof}

\begin{corollary}
\label{log}$g H^{1} (\lambda) =\gamma H^{1} (\mu)$ for some unimodular
$\gamma\Leftrightarrow\int_{\mathbb{T}}\left\vert \log g\right\vert
d\mu<\infty\text{.}$
\end{corollary}

\begin{proof}
Assume $g H^{1} (\lambda) =\gamma H^{1} (\mu)$, Since $1 \in H^{1} (\lambda)
,g \in g H^{1} (\lambda)$, $\exists\phi\in H^{1} (\mu)$ such that $g
=\gamma\phi$. Since $\phi\in H^{1} (\mu) ,\phi=\psi h$, where $\psi$ is an
inner function and $h$ is an outer function. Thus, $\int_{\mathbb{T}%
}\left\vert \log g\right\vert d\mu=\int_{\mathbb{T}}\log\vert g\vert d\mu
=\int_{\mathbb{T}}\log\vert h\vert d\mu<\infty$, since $h$ is an outer function.

Assume $\int_{\mathbb{T}}\left\vert \log g\right\vert d\mu<\infty,g$ and $\log
g \in L^{1} (\mu) ,g >0.$ Thus there exists an outer function $h \in H^{1}
(\mu)$, such that $\vert h\vert=\vert g\vert=g ,\left\vert h\right\vert =\phi
h ,\vert\phi\vert=1 ,g =\eta h$, Define $V :L^{1} (\lambda) \longrightarrow
L^{1} (\mu)$ by $V f =h f$, as in the theorem \ref{hlmd}, we have $h H^{1}
(\lambda) =H^{1} (\mu)$, so $g H^{1} (\lambda) =\eta h H^{1} (\lambda) =\eta
H^{1} (\mu)\text{.}$ Let\ $\gamma=\eta\text{,}$ then $g H^{1} (\lambda)
=\gamma H^{1} (\mu)\text{.}$
\end{proof}

We now get a Helson-Lowdenslager theorem when $\alpha=\Vert\cdot\Vert_{p
,\lambda}$ and $\log g$ $\in L^{1} \left(  \mu\right)  $.

\begin{corollary}
Suppose $1 \leq p <\infty$. If $W$ is a closed subspace of $L^{p} (\lambda)$
and $z W \subseteq W$, then either $W =\gamma H^{p} (\lambda)$ for some
unimodular function $\gamma$, or $W =\chi_{E} L^{p} (\lambda)$ for some Borel
subset $E$ of $\mathbb{T}$.
\end{corollary}

The following theorem shows the relation between $H^{\alpha},H^{1}(\lambda)$
and $L^{\alpha}$. This result parallels a result of Y Chen \cite{chen2}, which
is a key ingredient in her proof of her general Beurling theorem. However, her
result was for $H^{1}\left(  \mu\right)  $ instead of $H^{1}\left(
\lambda\right)  $.

\begin{theorem}
\label{N8} $H^{\alpha} =H^{1} (\lambda) \cap L^{\alpha}\text{.}$
\end{theorem}

\begin{proof}
Since $\alpha$ is continuous $c \Vert\cdot\Vert_{1 ,\lambda}$-dominating,
$\alpha$-convergence implies $\Vert\cdot\Vert_{1 ,\lambda}$-convergence, thus
\[
H^{\alpha} =\overline{H^{\infty}}^{\alpha} \subseteq\overline{H^{\infty}%
}^{\Vert\cdot\Vert_{1 ,\lambda}} =H^{1} (\lambda)\text{.}%
\]
Also,%
\[
H^{\alpha} =\overline{H^{\infty} (\lambda)}^{\alpha} \subset\overline
{L^{\infty}}^{\alpha} =L^{\alpha}\text{.}%
\]
Thus $H^{\alpha} \subseteq H^{1} (\lambda) \cap L^{\alpha}\text{.}$

Since $\alpha$-convergence implies $\Vert\cdot\Vert_{1,\lambda}$-convergence,
$H^{1}(\lambda)\cap$ $L^{\alpha}$ is an $\alpha$-closed subspace of
$L^{\alpha}$. Suppose $\varphi\in\left(  L^{\alpha}\right)  ^{\#}$ such that
$\varphi|H^{\infty}=0$. It follows from lemma \ref{dual} that there is a $w\in
L^{1}\left(  \lambda\right)  $ such that $wL^{\alpha}\subset L^{1}\left(
\lambda\right)  $ and such that, for every $f\in L^{\alpha}$,
\[
\varphi\left(  f\right)  =\int f\bar{\eta}wd\lambda=\int fwhd\mu\text{.}%
\]
Since $wL^{\alpha}\subset L^{1}\left(  \lambda\right)  $, we know that
$whL^{\alpha}\subset L^{1}\left(  \mu\right)  $. Since $\varphi|_{H^{\infty}%
}=0$, we have%
\[
\int_{\mathbb{T}}z^{n}hwd\mu=\varphi\left(  z^{n}\right)  =0
\]
for every integer $n\geq0.$ Thus $hw\in H_{0}^{1}\left(  \mu\right)  $.

Now suppose $f\in H^{1}(\lambda)\cap L^{\alpha}$. Then $hf\in H^{1}\left(
\mu\right)  $. We know that every function in $H^{1}\left(  \mu\right)  $ has
a unique inner-outer factorization. Thus we can write%
\[
hf=\gamma_{1}h_{1}%
\]
with $\gamma_{1}$ inner and $h_{1}$ outer. Moreover, since $hw\in H_{0}%
^{1}\left(  \mu\right)  $, we can write%
\[
\left(  hw\right)  \left(  z\right)  =z\gamma_{2}\left(  z\right)
h_{2}\left(  z\right)
\]
with $\gamma_{2}$ inner and $h_{2}$ outer. By lemma \ref{outer}, we can find
real harmonic functions $u,u_{1},u_{2}\in L^{1}\left(  \mu\right)  $ such that%
\[
h=e^{u+i\overline{u}},h_{1}=e^{u_{1}+i\overline{u}_{1}},\; \text{and}%
\;h_{2}=e^{u_{2}+i\overline{u}_{2}}\text{.}%
\]
Thus
\[
hfw=hfhw/h=z\gamma_{1}\gamma_{2}e^{\left(  u_{1}+u_{2}-u\right)  +i\left(
\overline{u}_{1}+\overline{u}_{2}-\overline{u}\right)  }\in H^{1}\left(
\mu\right)  \text{.}%
\]
It follows from lemma \ref{outer} that%
\[
\varphi\left(  f\right)  =\int_{\mathbb{T}}hfwd\mu=\left(  hfw\right)  \left(
0\right)  =0.
\]
Hence every continuous linear functional on $L^{\alpha}$ that annihilates
$H^{\alpha}$ also annihilates $H^{1}\left(  \lambda\right)  \cap L^{\alpha}$.
It follows from the Hahn-Banach theorem that $H^{1}\left(  \lambda\right)
\cap L^{\alpha}\subset H^{\alpha}\text{.}$
\end{proof}

The following result is a factorization theorem for $L^{\alpha}\text{.}$

\begin{theorem}
\label{factorization} If $k \in L^{\infty}$, $k^{ -1} \in$ $L^{\alpha}$, then
there is a unimodular function $u \in L^{\infty}$ and an outer function $s \in
H^{\infty}$ such that $k =u s$ and $s^{ -1} \in$ $H^{\alpha}$.
\end{theorem}

\begin{proof}
Recall that an outer function is uniquely determined by its absolute boundary
values, which are necessarily absolutely log integrable. Since $k^{ -1} \in$
$L^{\alpha} \subseteq L^{1} (\lambda)$, we know that $\left\Vert k\right\Vert
_{\infty} >0$. Thus $\log\left\vert k\right\vert \leq\log\left\Vert
k\right\Vert _{\infty} \in\mathbb{R}$. Moreover, $k^{ -1} \in$ $L^{\alpha}
\subseteq L^{1} (\lambda)$ implies $h k^{ -1} \in L^{1} \left(  \mu\right)
\text{,}$ so
\[
\log\left\vert h\right\vert -\log\left\vert k\right\vert =\log\left(
\left\vert h k^{ -1}\right\vert \right)  \leq\left\vert h k^{ -1}\right\vert
\text{.}%
\]
Hence%
\[
\log\left\vert h\right\vert -\left\vert h k^{ -1}\right\vert \leq
\log\left\vert k\right\vert \leq\log\left\Vert k\right\Vert _{\infty}\text{,}%
\]
and since $\log\left\vert h\right\vert \text{,}$ $\left\vert h k^{
-1}\right\vert $ and $\log\left\Vert k\right\Vert _{\infty}$ are in $L^{1}
\left(  \mu\right)  $, we see that $\log\left\vert k\right\vert \in L^{1}
\left(  \mu\right)  $. Therefore, by the first statement of lemma \ref{cases},
there is an outer function $s \in H^{1} \left(  \mu\right)  $ such that
$\left\vert s\right\vert =\left\vert k\right\vert $. It follows that $s \in
H^{\infty}$. Hence there is a unimodular function $u$ such that $k =u
s\text{.}$

We also know that%
\[
\left\vert \log\left\vert h k^{ -1}\right\vert \right\vert =\left\vert
\log\left(  \left\vert h\right\vert \right)  -\log\left\vert k\right\vert
\right\vert \leq\left\vert \log\left(  \left\vert h\right\vert \right)
\right\vert +\left\vert \log\left\vert k\right\vert \right\vert \in L^{1}
\left(  \mu\right)  \text{,}%
\]
so there exists an outer function $f \in H^{1} (\mu)$ such that $\vert k^{ -1}
h\vert=\vert f\vert$. Thus $s f$ is outer in $H^{1} \left(  \mu\right)  $ and
$\left\vert h\right\vert =\left\vert s f\right\vert $, so $h =e^{i t} s f$ for
some real number $t\text{.}$ Since $H^{1} (\mu) =h H^{1} (\lambda)$, we see
that there exists a function $f_{1} \in H^{1} (\lambda)$ such that $h f_{1} =f
=h \left(  e^{ -i t} s^{ -1}\right)  \text{.}$ It follows that $s^{ -1} =e^{i
t} f_{1} \in H^{1} \left(  \lambda\right)  $. Also, $\left\vert s^{
-1}\right\vert =\left\vert k^{ -1}\right\vert $, so $s^{ -1} \in L^{\alpha
}\text{.}$ It follows from Theorem \ref{N8} that $s^{ -1} \in H^{1} (\lambda)
\cap L^{\alpha} =$ $H^{\alpha}$.
\end{proof}

\begin{lemma}
\label{hift}If $M$ is a closed subspace of $L^{\alpha}$ and $z M \subseteq M$,
then $H^{\infty} M \subseteq M$.
\end{lemma}

\begin{proof}
Suppose $\varphi\in\left(  L^{\alpha}\right)  ^{\#}$ and $\varphi|_{M}=0.$ It
follows from lemma \ref{dual} that there is a $w\in L^{1}\left(
\lambda\right)  $ such that $wL^{\alpha}\subset L^{1}\left(  \lambda\right)  $
such that, for every $f\in L^{\alpha}$%
\[
\varphi\left(  f\right)  =\int_{\mathbb{T}}fw\bar{\eta}d\lambda=\int%
_{T}fwhd\mu\text{.}%
\]
Suppose $f\in M$. Then, for every integer $n\geq0$, we have $z^{n}f\in M$, so%
\[
0=\int_{\mathbb{T}}z^{n}fwhd\mu\text{.}%
\]
Since $fwh\in hL^{1}\left(  \lambda\right)  =L^{1}\left(  \mu\right)  $, it
follows that $fwh\in H_{0}^{1}\left(  \mu\right)  $. Thus if $k\in H^{\infty}%
$, we have%
\[
0=\int_{\mathbb{T}}kfwhd\mu=\varphi\left(  kf\right)  \text{.}%
\]
Hence every $\varphi\in\left(  L^{\alpha}\right)  ^{\#}$ that annihilates $M$
must annihilate $H^{\infty}M$. It follows from the Hahn-Banach theorem that
$H^{\infty}M\subset M\text{.}$
\end{proof}

We let $\mathbb{B} =\{f \in L^{\infty} :\Vert f\Vert_{\infty} \leq1\}$ denote
the closed unit ball in $L^{\infty} (\lambda)$.

\begin{lemma}
\label{coincide}Let $\alpha$ be a continuous norm on $L^{\infty}(\lambda)$,
then \newline(1) The $\alpha$-topology, the $\Vert\cdot\Vert_{2,\lambda}%
$-topology, and the topology of convergence in $\lambda$-measure coincide on
$\mathbb{B}$, \newline(2) $\mathbb{B}=\{f\in L^{\infty}(\lambda):\Vert
f\Vert_{\infty}\leq1\}$ is $\alpha$-closed.
\end{lemma}

\begin{proof}
For (1), since $\alpha$ is $c \Vert\cdot\Vert_{1 ,\lambda}$-dominating,
$\alpha$-convergence implies $\Vert\cdot\Vert_{1 ,\lambda}$-convergence, and
$\Vert\cdot\Vert_{1 ,\lambda}$-convergence implies convergence in measure.
Suppose $\{f_{n}\}$ is a sequence in $\mathbb{B}$, $f_{n} \rightarrow f$ in
measure and $\varepsilon>0.$ If $E_{n} =\{z \in\mathbb{T} :\vert f (z) -f_{n}
(z)\vert\geq\frac{\varepsilon}{2}\}$, then $\lim_{n \rightarrow\infty}(E_{n})
=0$. Since $\alpha$ is continuous, we have $\lim_{n \rightarrow\infty}%
\alpha(\chi_{E_{n}}) =0$, which implies that
\begin{align*}
\alpha(f_{n} -f)  &  =\alpha((f -f_{n}) \chi_{E_{n}} +(f -f_{n})
\chi_{\mathbb{T}\backslash E_{n}})\\
&  \leq\alpha((f -f_{n}) \chi_{E_{n}}) +\alpha((f -f_{n}) \chi_{\mathbb{T}%
\backslash E_{n}})\\
&  <\alpha((\vert f -f_{n}\vert) \chi_{E_{n}}) +\frac{\varepsilon}{2}
\leq\Vert f -f_{n}\Vert_{\infty} \alpha(\chi_{E_{n}}) +\frac{\varepsilon}{2}\\
&  \leq2 \alpha(\chi_{E_{n}}) +\frac{\varepsilon}{2}\text{.}%
\end{align*}
Hence $\alpha(f_{n} -f) \rightarrow0$ as $n \rightarrow\infty$. Therefore
$\alpha$-convergence is equivalent to convergence in measure on $\mathbb{B}$.
Since $\alpha$ was arbitrary, letting $\alpha=\Vert\cdot\Vert_{2 ,\lambda}$,
we see that $\Vert\cdot\Vert_{2 ,\lambda}$-convergence is also equivalent to
convergence in measure. Therefore, the $\alpha$-topology and the $\Vert
\cdot\Vert_{2 ,\lambda}$-topology coincide on $\mathbb{B}$.

For (2), suppose $\{f_{n}\}$ is a sequence in $\mathbb{B}$, $f\; \in
L^{\alpha}$ and $\alpha(f_{n} -f) \rightarrow0$. Since $\Vert f\Vert_{1
,\lambda} \leq\frac{1}{c} \alpha(f)$. it follows that $\Vert f_{n} -f\Vert_{1
,\lambda} \rightarrow0$, which implies that $f_{n} \rightarrow f$ in $\lambda
$-measure. Then there is a subsequence $f_{n_{k}}$such that $f_{n_{k}}
\rightarrow f$ a.e. $\left(  \lambda\right)  $. Hence $f \in\mathbb{B}%
\text{.}$
\end{proof}

The following theorem and its corollary relate the closed invariant subspaces
of $L^{\alpha}$ to the weak*-closed invariant subspaces of $L^{\infty}$.

\begin{theorem}
\label{density}Let $W$ be an $\alpha$-closed linear subspace of $L^{\alpha}$
and $M$ be a weak*-closed linear subspace of $L^{\infty}(\lambda)$ such that
$zM\subseteq M$ and $zW\subseteq W$. Then \newline(1) $M=\overline{M}^{\alpha
}\cap L^{\infty}(\lambda)$, \newline(2) $W\cap L^{\infty}(\lambda)$ is
weak*-closed in $L^{\infty}(\lambda)$, \newline(3) $W=\overline{W\cap
L^{\infty}(\lambda)}^{\alpha}\text{.}$
\end{theorem}

\begin{proof}
For (1), it is clear that $M\subset\overline{M}^{\alpha}\cap L^{\infty
}(\lambda)$. Assume, via contradiction, that $w\in\overline{M}^{\alpha}\cap
L^{\infty}(\lambda)$ and $w\notin M$. Since $M$ is weak*-closed, there is an
$F\in L^{1}(\lambda)$ such that $\int_{\mathbb{T}}Fwd\lambda\neq0\text{,}$ but
$\int_{\mathbb{T}}Frd\lambda=0$ for every $r\in M$. Since $k=\frac{1}%
{|F|+1}\in L^{\infty}(\lambda)$, $k^{-1}\in L^{1}(\lambda)\text{,}$ it follows
from theorem \ref{factorization}, that there is a $s\in H^{\infty}%
(\lambda),s^{-1}\in H^{1}(\lambda)$ and a unimodular function $u$ such that
$k=us$. Choose a sequence $\{s_{n}\}$ in $H^{\infty}(\lambda)$ such that
$\Vert s_{n}-s^{-1}\Vert_{1,\lambda}\rightarrow0$. Since $sF=\overline
{u}kF=\overline{u}\frac{F}{|F|+1}\in L^{\infty}(\lambda)$, we can conclude
that $\Vert s_{n}sF-F\Vert_{1,\lambda}=\Vert s_{n}sF-s^{-1}sF\Vert_{1,\lambda
}\leq\Vert s_{n}-s^{-1}\Vert_{1,\lambda}\Vert sF\Vert_{\infty}\rightarrow0.$
For each $n\in\mathbb{N}$. For every $r\in M$, from lemma \ref{hift}, we know
that $s_{n}sr\in H^{\infty}(\lambda)M\subset M$. Hence%
\[
\int_{\mathbb{T}}rs_{n}sFd\lambda=\int_{\mathbb{T}}s_{n}srFd\lambda=0,\forall
r\in M\text{.}%
\]
\newline Suppose $r\in\overline{M}^{\alpha}\text{.}$ Then there is a sequence
$r_{m}$ in $M$ such that $\alpha(r_{m}-r)\rightarrow0$ as $m\rightarrow
\infty\text{.}$ For each $n\in\mathbb{N}\mathbb{N}$, it follows from
$s_{n}sF\in H^{\infty}(\lambda)L^{\infty}(\lambda)$ that \newline%
\begin{align*}
|\int_{\mathbb{T}}rs_{n}sFd\lambda-\int_{\mathbb{T}}r_{m}s_{n}sFd\lambda|  &
\leq\int_{\mathbb{T}}|(r-r_{m})s_{n}sF|d\lambda\\
&  \leq\Vert s_{n}sF\Vert_{\infty}\int_{\mathbb{T}}|r-r_{m}|d\lambda=\Vert
s_{n}sF\Vert_{\infty}\Vert r-r_{m}\Vert_{1,\lambda}\\
&  \leq\Vert s_{n}sF\Vert_{\infty}\alpha(r-r_{m})\rightarrow0.\\
\int_{\mathbb{T}}rs_{n}sFd\lambda &  =\lim_{m\rightarrow0}\int_{\mathbb{T}%
}r_{m}s_{n}sFd\lambda=0,\forall r\in\overline{M}^{\alpha}\text{.}%
\end{align*}
In particular, $w\in\overline{M}^{\alpha}\cap L^{\infty}(\lambda)$ implies
that
\[
\int_{\mathbb{T}}s_{n}sFwd\lambda=\int_{\mathbb{T}}ws_{n}sFd\lambda=0.
\]
Hence,
\begin{align*}
0  &  \neq|\int_{\mathbb{T}}Fwd\lambda|\leq\lim_{n\rightarrow\infty}%
|\int_{\mathbb{T}}Fw-s_{n}sFwd\lambda|+\lim_{n\rightarrow\infty}%
|\int_{\mathbb{T}}s_{n}sFwd\lambda|\\
&  \leq\lim_{n\rightarrow\infty}\Vert F-s_{n}sF\Vert_{1,\lambda}\Vert
w\Vert_{\infty}+0=0.
\end{align*}
\newline We get a contradiction. Hence $M=\overline{M}^{\alpha}\cap L^{\infty
}(\lambda)\text{.}$

For (2), to prove $W\cap L^{\infty}(\lambda)$ is weak*-closed in $L^{\infty
}(\lambda)$, using the Krein-Smulian theorem, we only need to show that $W\cap
L^{\infty}(\lambda)\cap\mathbb{B}$, i.e., $W\cap\mathbb{B}$, is weak*-closed.
By lemma \ref{coincide}, $W\cap\mathbb{B}$ is $\alpha$-closed. Since $\alpha$
is $c\Vert\cdot\Vert_{1,\lambda}$-dominating, it follows from the lemma
\ref{coincide}, $W\cap\mathbb{B}$ is $\Vert\cdot\Vert_{2,\lambda}$ closed. The
fact that $W\cap\mathbb{B}$ is convex implies $W\cap\mathbb{B}$ is closed in
the weak topology on $L^{2}(\lambda)$. If $\{f_{\lambda}\}$ is a net in
$W\cap\mathbb{B}$ and $f_{\lambda}\rightarrow f$ weak* in $L^{\infty}%
(\lambda)$, then, for every $w\in L^{1}(\lambda),\int_{\mathbb{T}}(f_{\lambda
}-f)wd\lambda\rightarrow0$. Since $L^{2}(\lambda)\subset L^{1}(\lambda)$,
$f_{\lambda}\rightarrow f$ weakly in $L^{2}(\lambda)$, so $f\in W\cap
\mathbb{B}$. Hence $W\cap\mathbb{B}$ is weak*-closed in $L^{\infty}(\lambda)$.
\newline For (3), since $W$ is $\alpha$-closed in $L^{\alpha}$, it is clear
that $W\supset\overline{W\cap L^{\infty}(\lambda)}^{\alpha}$, suppose $f\in W$
and let $k=\frac{1}{|f|+1}$. Then $k\in L^{\infty}(\lambda)$, $k^{-1}\in$
$L^{\alpha}\text{.}$ It follows from Theorem \ref{factorization} that there is
a $s\in H^{\infty}(\lambda),s^{-1}\in$ $H^{\alpha}$ and an unimodular function
$u$ such that $k=us$, so $sf=\overline{u}ks=\overline{u}\frac{f}{|f|+1}\in
L^{\infty}(\lambda)$. There is a sequence $\{s_{n}\}$ in $H^{\infty}(\lambda)$
such that $\alpha(s_{n}-s^{-1})\rightarrow0$. For each $n\in\mathbb{N}$, it
follows from the lemma \ref{hift} that $s_{n}sf\in H^{\infty}(\lambda
)H^{\infty}(\lambda)W\subset W$ and $s_{n}sf\in H^{\infty}(\lambda)L^{\infty
}(\lambda)\subset L^{\infty}(\lambda)$, which implies that $s_{n}sf$ is a
sequence in $W\cap L^{\infty}(\lambda)$, $\alpha(s_{n}sf-f)\leq\alpha
(s_{n}-s^{-1})\Vert sf\Vert_{\infty}\rightarrow0$. Thus $f\in\overline{W\cap
L^{\infty}(\lambda)}^{\alpha}$. Therefore $W=\overline{W\cap L^{\infty
}(\lambda)}^{\alpha}$.
\end{proof}

\begin{corollary}
\label{w*}A weak*-closed linear subspace $M$ of $L^{\infty} (\lambda)$
satisfies $z M \subset M$ if and only if $M =\varphi H^{\infty} (\lambda)$ for
some unimodular function $\varphi$ or $M =\chi_{E} L^{\infty} (\lambda)$, for
some Borel subset $E$ of $\mathbb{T}$.
\end{corollary}

\begin{proof}
If $M =\varphi H^{\infty} (\lambda)$ for some unimodular function $\varphi$ or
$M =\chi_{E} L^{\infty} (\lambda)$, for some Borel subset $E$ of $\mathbb{T}$,
clearly, a weak*-closed linear subspace $M$ of $L^{\infty} (\lambda)$ with $z
M \subset M$. Conversely, since $z M \subset M$, and we have $z \overline
{M}^{\Vert\cdot\Vert_{2 ,\lambda}} \subset\overline{M}^{\Vert\cdot\Vert_{2
,\lambda}}$. Hence by Beurling-Helson-Lowdenslager theorem for $\Vert
\cdot\Vert_{2 ,\lambda}$, we consider either $\overline{M}^{\Vert\cdot\Vert_{2
,\lambda}} =\varphi H^{2} (\lambda)$ for some unimodular function $\varphi$,
then $M =\overline{M}^{\Vert\cdot\Vert_{2 ,\lambda}} \cap L^{\infty} (\lambda)
=\varphi H^{2} (\lambda) \cap L^{\infty} (\lambda)$; or $\overline{M}%
^{\Vert\cdot\Vert_{2 ,\lambda}} =\chi_{E} L^{2} (\lambda)$, for some Borel
subset $E$ of $\mathbb{T}$, in this case, $M =\overline{M}^{\Vert\cdot\Vert_{2
,\lambda}} \cap L^{\infty} (\lambda) =\chi_{E} L^{2} (\lambda) \cap L^{\infty}
(\lambda) =\chi_{E} L^{\infty} (\lambda)$, i.e., $M =\chi_{E} L^{\infty}
(\lambda)$.
\end{proof}

Now we obtain our main theorem, which extends the Chen-Beurling
Helson-Lowdenslager Theorem.

\begin{theorem}
Suppose $\mu$ is Haar measure on $\mathbb{T}$ and $\alpha$ is a continuous
normalized gauge norm on $L^{\infty}(\mu)$. Suppose also that $c>0$ and
$\lambda$ is a probability measure that is mutually absolutely continuous with
respect to $\mu$ such that $\alpha$ is $c\left\Vert \text{{}}\right\Vert
_{1,\lambda}$-dominating and $\log\left\vert d\lambda/d\mu\right\vert \in
L^{1}\left(  \mu\right)  \text{.}$ Then a closed linear subspace $W$ of
$L^{\alpha}(\mu)$ satisfies $zW\subset W$ if and only if either $W=\varphi
H^{\alpha}(\mu)$ for some unimodular function $\varphi$, or $W=\chi
_{E}L^{\alpha}(\mu)$, for some Borel subset $E$ of $\mathbb{T}$. If $0\neq
W\subset H^{\alpha}(\mu)$, then $W=\varphi H^{\alpha}(\mu)$ for some inner
function $\varphi$.
\end{theorem}

\begin{proof}
Recall that $L^{\infty} (\mu) =L^{\infty} (\lambda) ,L^{\alpha} (\mu)
=L^{\alpha} (\lambda)$ and $H^{\alpha} (\mu) =H^{\alpha} (\lambda)\text{.}$The
only if part is obvious. Let $M =W \cap L^{\infty} (\lambda)$, and in the
theorem \ref{thm(1)}, we proved that there exists a measure $\lambda$ such
that $\lambda\ll\mu$ and $\mu\ll\lambda$ and there exists $c >0$, $\forall f
\in L^{\infty} (\mu) =L^{\infty} (\lambda) ,\alpha(f) \geq$ $c \Vert f\Vert_{1
,\lambda}$. i.e., $\alpha$ is a continuous $c \Vert\cdot\Vert_{1 ,\lambda}%
$-dominating normalized gauge norm on $L^{\infty} (\lambda)$. It follows from
the (2) in theorem \ref{density} that $M$ is weak* closed in $L^{\infty}
(\lambda)$. Since $z W \subset W$, it is easy to check that $z M \subset M$.
Then by the corollary \ref{w*}, we can conclude that either $M =\varphi
H^{\infty} (\lambda)$ for some unimodular function $\varphi$ or $M =\chi_{E}
L^{\infty} (\lambda)$, for some Borel subset $E$ of $\mathbb{T}$. By the (3)
in theorem \ref{density}, if $M =\varphi H^{\infty} (\lambda)\text{,}$ $W
=\overline{W \cap L^{\infty} (\lambda)}^{\alpha} =\overline{M}^{\alpha}
=\overline{\varphi H^{\infty} (\lambda)}^{\alpha} =\varphi H^{\alpha} =\varphi
H^{\alpha} (\mu)$, for some unimodular function $\varphi$. If $M =\chi_{E}
L^{\infty} (\lambda)\text{,}$ $W =\overline{W \cap L^{\infty} (\lambda
)}^{\alpha} =\overline{M}^{\alpha} =\overline{\chi_{E} L^{\infty} (\lambda
)}^{\alpha} =\chi_{E} L^{\alpha} =\chi_{E} L^{\alpha} (\mu)$, for some Borel
subset $E$ of $\mathbb{T}$. The proof is complete.
\end{proof}

\section{Which $\alpha$'s have a good $\lambda$?}

In the preceding section we proved a version of Beurling's theorem for
$L^{\alpha}$ when there is a probability measure $\lambda$ on $\mathbb{T}$
that is mutually absolutely continuous with respect to $\mu$, such that
$\alpha$ is $c \Vert\cdot\Vert_{1 ,\lambda}$-dominating and $d \lambda/d \mu$
is log-integrable with respect to $\mu$. How do we tell when such a good
$\lambda$ exists. Suppose $\rho$ is a probability measure on $\mathbb{T}$ that
is mutually absolutely continuous with respect to $\mu$ such that%
\[
\int_{\mathbb{T}}\log\left(  d \rho/d \mu\right)  d\mu= -\infty\text{.}%
\]
Here are some useful examples.

\begin{example}
Let $\alpha=\frac{1}{2} \left(  \Vert\cdot\Vert_{1 ,\mu} +\Vert\cdot\Vert_{1
,\rho}\right)  $. Then $\alpha$ is a continuous gauge norm. If we let
$\lambda_{1} =\rho$ and $\lambda_{2} =\mu$ we see that $\alpha\geq\frac{1}{2}
\lambda_{k}$ for $k =1 ,2$ and%
\[
\int_{\mathbb{T}}\left\vert \log\left(  d \lambda_{k}/d \mu\right)
\right\vert d\mu=\left\{
\begin{array}
[c]{cc}%
\infty & \; \text{if}\;k =1\\
0 & \; \text{if}\;k =2
\end{array}
\right.  \text{.}%
\]
Hence there is both a bad choice of $\lambda$ and a good choice.
\end{example}

\begin{example}
Suppose $\rho$ is as in the preceding example and let $\alpha=\Vert\cdot
\Vert_{1 ,\rho}\text{.}$ Suppose $\lambda$ is a probability measure that is
mutually absolutely continuous with respect to $\mu$ and
\[
\Vert\cdot\Vert_{1 ,\rho} =\alpha\geq c \Vert\cdot\Vert_{1 ,\lambda}\;
\text{for some constant}\;c\text{.}%
\]
It follows that $d \lambda/d \rho\leq c$ a.e., and thus
\[
\int_{\mathbb{T}}\log\left(  d \lambda/d \mu\right)  d\mu=\int_{\mathbb{T}%
}\log\left(  d \lambda/d \rho\right)  d\mu+\int_{\mathbb{T}}\log\left(  d
\rho/d \mu\right)  d\mu\leq\log\varepsilon+\left(  -\infty\right)  =
-\infty\text{.}%
\]
In this case there is no good $\lambda$.
\end{example}

\section{A special case.}

Suppose $\lambda$ is any probability measure that is mutually absolutely
continuous with respect to $\mu$ and $\alpha=\Vert\cdot\Vert_{p,\lambda}$ for
some $p$ with $1\leq p<\infty$. Assume $\lambda$ is bad, i.e., $\int%
_{\mathbb{T}}\left\vert \log\frac{d\lambda}{d\mu}\right\vert d\mu
=\infty\text{.}$ In this case, we define a bijective isometry mapping
$U:L^{p}(\lambda)\rightarrow L^{p}(\mu)$ by $Uf=g^{\frac{1}{p}}f\text{.}$ Let
$H^{p}(\lambda)$ be the $\alpha$-closure of all polynomials, then
$H^{p}(\lambda)$ is a closed subspace of $L^{p}(\lambda)$ and $zH^{p}%
(\lambda)\subseteq H^{p}(\lambda)\text{.}$ Therefore, $g^{\frac{1}{p}}%
H^{p}(\lambda)$ is a $z$-invariant closed subspace of $L^{p}(\mu)\text{.}$ By
Beurling-Helson-Lowdenslager theorem, we have
\[
g^{\frac{1}{p}}H^{p}(\lambda)=\chi_{E}L^{p}(\mu)\; \text{for some Borel
set}\;E\subseteq\mathbb{T},\; \text{or}\; \varphi H^{p}(\mu),\; \text{where}\;
\left\vert \varphi\right\vert =1.
\]
If $g^{\frac{1}{p}}H^{p}(\lambda)=\chi_{E}L^{p}(\mu)\text{,}$ then
$H^{p}(\lambda)=L^{p}(\lambda)\text{,}$ in this case, $\varphi H^{p}%
(\lambda)=\varphi L^{p}(\lambda)\text{,}$ where $\left\vert \varphi\right\vert
=1.$ If $M_{0}=\frac{1}{g^{1/p}}H^{p}(\mu),$ then $M_{0}$ is a proper
$z$-invariant closed subspace of $L^{p}(\lambda)\text{,}$ and $M_{0}\neq
\chi_{E}L^{p}(\lambda)\text{.}$ Therefore, Beurling-Helson-Lowdenslager
theorem is not true for this case. However, we have the following theorem

\begin{theorem}
Suppose $\lambda$ is any probability measure that is mutually absolutely
continuous with respect to $\mu$ and $\alpha=\Vert\cdot\Vert_{p,\lambda}$ for
some $p$ with $1\leq p<\infty$. Also assume $\int_{\mathbb{T}}\left\vert
\log\frac{d\lambda}{d\mu}\right\vert d\mu=\infty\text{.}$ If $M$ is a closed
subspace of $L^{\alpha}(\lambda)\text{,}$ then $zM\subseteq M$ if and only if
\newline(1) $M=\varphi M_{0}$ for some unimodular function $\varphi$, where
$M_{0}=\frac{1}{g^{1/p}}H^{p}(\mu)\text{,}$ or \newline(2) $M=\chi
_{E}L^{\alpha}(\lambda)$ for some Borel subset $E$ of $\mathbb{T}$.
\end{theorem}


\begin{thebibliography}{99}                                                                                               %


\bibitem {B}A. Beurling, \emph{On two problems concerning linear
transformations in Hilbert space}, Acta Math. \textbf{81} (1949) 239--255.

\bibitem {chen}Y. Chen, \emph{Lebesgue and Hardy spaces for symmetric norms
}$I$, arXiv: 1407.7920 [math. OA] (2014).

\bibitem {chen2}Y. Chen, \emph{A general Beurling-Helson-Lowdenslager theorem
on the disk}, arXiv: 1501.05718v1 [math. OA] (2015).

\bibitem {duren}P. Duren, \emph{Theory of }$H^{p}$\emph{\ spaces}, Academic
Press, New York, 1970.

\bibitem {halmos}P. Halmos, \emph{Shifts on Hilbert spaces}, J. Reine Angew.
Math. 208(1961) 102-112.

\bibitem {helson}H. Helson, \emph{Lectures on Invariant Subspace, Academic
Press}, New York-London, 1964.

\bibitem {H.L}H. Helson and D. Lowdenslager, \emph{Prediction theory and
Fourier series in several variables}, Acta Math. 99 (1958) 519-540.

\bibitem {hoffman}K. Hoffman, \emph{Analytic functions and logmodular Banach
algebras}, Acta Math. 108(1962) 271-317.

\bibitem {niko}N. K. Nikolski\i, \emph{Treatise on the shift operator,
Spectral function theory}, With an appendix by S. V. Hruscev and V. V. Peller,
Translated from the Russian by Jaak Peetre, Grundlehren der Mathematischen
Wissenschaften, 273. Springer-Verlag, Berlin, 1986.

\bibitem {srinivasan}T. P. Srinivasan\emph{, Simply invariant subspaces},
Bull. Amer. Math. Soc. 69 (1963) 706--709.
\end{thebibliography}
\end{document}